\PassOptionsToPackage{hyphens}{url}
\documentclass{amsart}
\usepackage{lmodern}
\usepackage[capitalize]{cleveref}
\usepackage{amssymb,amsmath,amsthm}
\usepackage{ifxetex,ifluatex}
\usepackage{setspace}
\usepackage{tikz-cd}
\usepackage[T1]{fontenc}
\usepackage{textcomp} 
\usepackage{todonotes}
\usepackage{xcolor}
\providecommand{\tightlist}{%
  \setlength{\itemsep}{0pt}\setlength{\parskip}{0pt}}
\usepackage{amsthm}
\usepackage{thmtools}
\usepackage[
backend=biber,
            style=alphabetic,
            url=false,
            maxnames=10,
            minnames=5
]{biblatex}

\newcommand{\A}{\mathcal{A}}

\newcommand{\B}{\mathcal{B}}
\newcommand{\M}{\mathcal{M}}

\renewcommand{\S}{\mathcal{S}}

\newcommand{\mc}[1]{\mathcal{#1}}
\newcommand{\ra}{\rightarrow}
\newcommand{\Ra}{\Rightarrow}

\renewcommand{\phi}{\varphi}

\newcommand{\scrA}{\mathcal{A}}
\DeclareMathOperator{\restrict}{\upharpoonright}

\DeclareMathOperator{\ar}{\leadsto}
\DeclareMathOperator{\concat}{^\smallfrown}
\DeclareMathOperator{\enters}{\searrow}
\declaretheorem[name={Theorem}]{theorem}
\declaretheorem[name={Proposition}, sibling=theorem]{proposition}
\declaretheorem[name={Lemma}, sibling=theorem]{lemma}

\declaretheorem[name={Claim}, numberwithin=theorem]{claim}
\declaretheorem[name={Definition}, style=definition]{definition}
\declaretheorem[name={Corollary}, sibling=theorem]{corollary}
\declaretheorem[name={Question}]{question}
\ifluatex
\usepackage{fontspec}
  \usepackage{selnolig}  
\fi
\makeatletter
\def\paragraph{\@startsection{paragraph}{4}%
  \z@\z@{-\fontdimen2\font}%
  {\normalfont\itshape}}
\makeatother
\title{Degrees of categoricity and treeable degrees}
\author{Barbara F. Csima \and Dino Rossegger}
\date{\today}

\address{Department of Pure Mathematics, University of Waterloo, Waterloo, Ontario
}
\email{csima@uwaterloo.ca}

\address{Department of Mathematics, University of California, Berkeley {\normalfont and} Institute of Discrete Mathematics and Geometry, Technische Universit\"at Wien}
\email{dino@math.berkeley.edu}

\thanks{Csima is partially supported by an NSERC Discovery Grant. 
The work of the second author was supported by the European Union's Horizon 2020 Research and Innovation Programme under the Marie Sk\l{}odowska-Curie grant agreement No. 101026834 — ACOSE}

\begin{document}
\begin{abstract}
  We give a characterization of the strong degrees of categoricity of
  computable structures greater or equal to
  $\mathbf 0''$. They are precisely the \emph{treeable} degrees---the least
  degrees of paths through computable trees---that compute $\mathbf 0''$.
  As a corollary, we obtain several new examples of degrees of
  categoricity. Among them we show that every degree $\mathbf d$ 
  with $\mathbf 0^{(\alpha)}\leq \mathbf d\leq \mathbf 0^{(\alpha+1)}$ for $\alpha$
  a computable ordinal greater than $2$ is the strong degree of
  categoricity of a rigid structure. 
  Using quite different techniques we show
  that every degree $\mathbf d$ with $\mathbf 0'\leq \mathbf d\leq \mathbf
  0''$ is the strong degree of categoricity of a structure. Together with the above
  example this answers a question of Csima and Ng.
  To complete the picture we show that there is a degree $\mathbf d$ with
  $\mathbf 0'<
  \mathbf d< \mathbf 0''$ that is not the degree of categoricity of a rigid
  structure.
\end{abstract}
\maketitle

Two isomorphic copies of a mathematical structure share the same structural
properties and, thus, one usually considers structures up to isomorphism.
However, Fr\"ohlich and Shepherdson~\cite{frohlich1956} and, independently,
Malt'sev~\cite{maltsev1962} showed that isomorphic copies of a structure can
behave quite differently with respect to their algorithmic properties.
They produced two isomorphic computable fields $\mathcal K_1$ and $\mathcal K_2$
such that $\mathcal K_1$ has a computable transcendence basis, while $\mathcal
K_2$ fails to have a computable transcendence basis. This leads to the
conclusion that there cannot be a computable
function that is an isomorphism between these two fields~\cite[Corollary 5.51]{frohlich1956}.

Two computable isomorphic structures $\A$ and $\B$ that are computably
isomorphic have the same computability theoretic properties, and if $\A$ has
computable isomorphisms between all its computable isomorphic copies then $\A$
is said to be \emph{computably categorical}. As we have seen, even natural
structures fail to be computably categorical, and often one would like to know how far
apart the computable copies of these structures can be with respect to their computability
theoretic properties. This is best captured by measuring the Turing complexity
of the isomorphisms between computable copies. Towards this, Fokina, Kalimullin, and Miller
introduced the following notion~\cite{fokina2010}.
\begin{definition}
  Let $\tau$ be a computable vocabulary, $\A$ be a computable $\tau$-structure,
  and let $(\B_e)_{e\in\omega}$ be an enumeration of all computable
  $\tau$-structures. The \emph{categoricity spectrum of $\A$}
  is the set
  \[CatSpec(\A)=\bigcap_{e\in \omega: \B_e\cong \A}\{deg(X): (\exists
  f:\A\cong\B_e) X\geq_T f\}.\]
  If $\mathbf d\in CatSpec(\A)$ is a least element, then $\mathbf d$ is called
  the \emph{degree of categoricity of $\A$}.
\end{definition}
All known degrees of categoricity actually have the following stronger property.

\begin{definition}\cite{fokina2010} A degree of categoricity ${\bf d}$ is a
	\emph{strong} degree of categoricity if there is a structure
	$\scrA$ with computable copies $\scrA_0$ and $\scrA_1$ such
	that ${\bf d}$ is the degree of categoricity for $\scrA$, and
	every isomorphism $f: \scrA_0 \rightarrow \scrA_1$ satisfies
	deg$(f) \geq {\bf d}$.
\end{definition}

The study of this notion has been one of the most active areas in computable
structure theory in the last decade. See~\cite{franklin2017} for a survey of
developments until 2017. One of the main goals in the area is to obtain
a characterization of the Turing degrees that are degrees of categoricity.

Fokina, Kalimullin, and Miller~\cite{fokina2010} showed that every degree
$\mathbf d$ d-c.e.\ in and above
$\mathbf 0^{(n)}$ for some $n\in\omega$ is a strong degree of categoricity and that $\mathbf 0^{(\omega)}$
is a degree of categoricity. Csima, Franklin, and Shore~\cite{csima2013}
generalized these results to the hyperarithmetic hierarchy, showing that both
$\mathbf 0^{(\alpha)}$ for $\alpha$ a computable limit ordinal and every degree $\mathbf d$ d-c.e.\
in and above $\mathbf 0^{(\alpha)}$ for $\alpha $ computable successor ordinal
are strong degrees of categoricity. The former result was later improved by Csima,
Deveau, Harrison-Trainor, and Mahmoud~\cite{csima2018a} to degrees c.e.\ in and
above $\mathbf 0^{(\lambda)}$ for $\lambda$ limit ordinals.

It is known that not every Turing degree is a degree of categoricity. Csima,
Franklin, and Shore~\cite{csima2013} showed that every degree of categoricity must be
hyperarithmetic and Anderson and Csima~\cite{anderson2012} provided several
examples of degrees that are not degrees of categoricity. For example, they
obtained a $\Sigma^0_2$ degree that is not a degree of categoricity.

Recently, Csima and Ng~\cite{csima} showed that every $\Delta^0_2$ degree is
a strong degree of categoricity. They asked whether for every computable ordinal
$\alpha$, every Turing degree $\mathbf d$,
$\mathbf 0^{(\alpha)}\leq \mathbf d\leq \mathbf 0^{(\alpha+1)}$ is a degree of
categoricity.

The main goal of this article is a characterization of the strong degrees of categoricity on the cone above
$\mathbf 0''$. We characterize these degrees by showing that they are exactly the degrees of Turing-least
paths through computable trees in $\omega^{\omega}$
(\cref{cor:dgcatifftreeable}). Using classical results about $\Pi^0_1$ function
singletons we then obtain that every degree $\mathbf d$ with $\mathbf
0^{(\alpha)}\leq \mathbf d\leq\mathbf 0^{(\alpha+1)}$ for $\alpha$ a computable
ordinal greater or equal to $2$ is the degree of categoricity of a rigid
structure. Building on a construction by Csima and Ng~\cite{csima} that showed that every
$\Delta^0_2$ degree is a degree of categoricity we complete the picture by
showing that every degree
$\mathbf d$, $\mathbf 0'<\mathbf d< \mathbf 0''$ is a degree of
categoricity. We thus obtain a positive answer to their first question
(\cref{thm:delta2relativized}). We also obtain more exotic examples. Csima and Stephenson \cite{CS2019} exhibited an example of a degree of categoricity that does not belong to an interval of the form $[\mathbf 0^{(\alpha)}$,$\mathbf
0^{(\alpha+1)}$] for any computable ordinal $\alpha$, but their degree \emph{is}
computable from $\mathbf{0''}$. Building on work of
Harrington~\cite{harrington1976a,harrington1976}, we obtain for every
computable ordinal $\alpha$, a degree of
categoricity that is not between $\mathbf 0^{(\gamma)}$ and $\mathbf
0^{(\gamma+1)}$ for any computable ordinal $\gamma$ and not computable from
$\mathbf 0^{(\alpha)}$ (\cref{cor:nonarithcompd}).

Most of the above results are corollaries of our main theorem whose proof is a modification of a recent construction of Turetsky that coded paths through trees into the automorphisms of
a structure~\cite{turetsky2020}. In order to state it, we need to recall a bit
of notation. Given a tree $T\subseteq \omega^{<\omega}$ we denote by
$[T]\subseteq \omega^\omega$ the set of paths through $T$. For two sets
$P,Q\subseteq \omega^\omega$ we say that $P$ is Muchnik reducible to Q,
$P\leq_w Q$, if for every $q\in Q$ there is $p\in P$ with $p\leq_T q$. We can
computably translate elements of $\omega^\omega$ to elements of $2^\omega$ and
sets of natural numbers and thus will use Muchnik reducibility to compare sets
of these types. 
The \emph{computable dimension} of a structure is the number of computable
copies that are not computably isomorphic. This number is either finite or
$\omega$ and constructions of examples of finite computable
dimension usually involve heavy computability theoretic
machinery. The first example of such a structure was given by
Goncharov~\cite{goncharov1980}. The following is our main theorem from which
most other results in this article are derived.
\begin{theorem}\label{thm:rigiddgcat}
  Let $T\subseteq \omega^{<\omega}$ be a computable tree such that $[T]\geq_w
  \{\emptyset''\}$. Then there is a computable structure $\S_1$ with computable
  dimension $2$ such that
  \[ CatSpec(\S_1)=\{ \deg(X): \{X\}\geq_w [T] \}.\]
  Furthermore, if $[T]$ is a singleton, then $\S_1$ is rigid.
\end{theorem}

\cref{thm:rigiddgcat} is proven in \cref{sec:paths} and its corollaries
are derived in \cref{sec:treeable}. In \cref{sec:csimang} we modify the
construction of Csima and Ng to show that every degree $\mathbf d$ with
$\mathbf 0'\leq\mathbf d\leq \mathbf 0''$ is the degree of categoricity of
a structure.
At last, in \cref{sec:bazhyama} we generalize a construction by Bazhenov and
Yamaleev~\cite{bazhenov2017d} to show that there is a Turing degree $\mathbf
d$, $\mathbf 0'\leq\mathbf d\leq\mathbf 0''$ that is not the degree of
categoricity of a rigid structure. This shows that the lower bound in \cref{thm:rigiddgcat} can
not be improved.

\section{Characterizing degrees of categoricity above $\mathbf 0''$}
\subsection{Degrees of categoricity and paths}\label{sec:paths}
Recall that for two sets $X,Y\subseteq \omega^\omega$, $X$ is \emph{Muchnik
reducible} to $Y$, $X\leq_w Y$ if for every $y\in Y$ there is $x\in X$ such
that $x\leq_T y$.
Turetsky proved that given a computable tree
$T\subseteq \omega^{<\omega}$ there is a computable, computably categorical
structure $\S$ such that the paths of $T$ and the non-trivial automorphisms of
$\S$ are Muchnik equivalent modulo $\mathbf 0''$~\cite{turetsky2020}. In other words,
\[ \{ \emptyset''\oplus f: f\in [T]\} \equiv_w \{\emptyset'' \oplus \nu :\nu\in
Aut(\S)\setminus id\}.\]
Turetsky also exhibited how the structure $\S$ can be adapted to obtain
a structure $\S_1$ that is not hyperarithmetically categorical and has
computable dimension $2$. For this structure, it is the case that the
isomorphisms between the two copies witnessing the computable dimension are
Muchnik equivalent to the paths through $T$ modulo $\mathbf 0''$.

If one can eliminate the $\mathbf 0''$ in these results one obtains a coding
technique that allows the coding of paths through trees into categoricity
spectra of structures of computable dimension $2$. If $T$ has a unique path
$f$, i.e., $f$ is a $\Pi^0_1$ function singleton, then an analysis of Turetsky's
construction shows that the structure $\S_1$ obtained from $T$ is rigid. 

Our first result improves on Turetsky's by eliminating $\emptyset''$ on the right.
We obtain this result by adding requirements to Turetsky's construction with
the aim of coding $\emptyset''$ into the presentation of $\S$. 
\begin{lemma}\label{lem:pathsinaut}
Let $T\subseteq \omega^{<\omega}$ be a computable tree. Then there is a computable, computably categorical structure $\S$ such that
\[\{\nu: \nu\in (Aut(\S)-\{id\})\}\equiv_w \{ f\oplus \emptyset'': f\in [T]\}.\]
In particular, $|[T]|=1$ if and only if $|Aut(\S)-\{id\}|=1$.
\end{lemma}
\cref{lem:pathsinaut} is the main ingredient to our proof of
\cref{thm:rigiddgcat}. Its proof is a modification of Turetsky's infinite injury construction~\cite[Theorem 2]{turetsky2020}. 
We give a full description of the structure we are going to construct and the
construction, but only sketch
the verification in the sense that we prove that our construction works and
gives the desired results if Turetsky's initial construction works. Thus, the
reader is encouraged to read this together with Turetsky's proof for a full
verification.
\begin{proof}[Proof sketch.]
  Given a tree $T$ the vocabulary of $\S$ consists of a unary relations $U$,
  $(W_\sigma)_{\sigma\in\omega^{<\omega}}$, and $(S_n)_{n\in\omega}$ ($V_n$ in
  Turetsky's construction), binary relations
  $P$ and $(E_n)_{n\in\omega}$, and a unary function $f$. We denote by
  $[\omega]^{<\omega}$ the set of all finite subsets of the natural numbers. The universe of $\S$
  will be $[\omega]^{<\omega}\times\omega^{<\omega}\sqcup C$ where $C$ is an
  infinite computable set. The relation symbols $U$ and $f$ and the set $C$ are
  used to help with the following issue. During the construction we will want
  $S_n(x)$ to hold for larger and larger $n$ on elements of
  $[\omega]^{<\omega}\times \omega$. The issue is, that this would result in
  a c.e.\ structure and not a computable structure. To overcome this we do not
  define the $S_n$ directly on the $x$ but rather using an element $y\in C$
  that is associated to $x$ by $f$, i.e., $f(y)=x$. Formally, we will have
  the following.
  
  The relation $U$ is used to identify $C$, i.e.
  $U^\S=C$. We will have that ${S_n}^\S\cap \overline{C}=\emptyset$ while none of
  the relations $W_\sigma$, $E_n$ and $P$ hold on any element in $C$.
  We will also delcare $f^\S(x)=x$ for all $x\not\in C$ and for each $x\in C$
  $f(x)\not\in C$. Furthermore, for each $x\in C$ there will be
  a unique $n$ such that $S_n(x)$ holds. During the
  construction we will ignore $C$, $U$ and $f$ and simply declare that $S_n(x)$
  holds for some element $x$ in $[\omega]^{<\omega}\times \omega^{<\omega}$. What
  this means is that we pick an unused element $y\in C$, set $f(y)=x$ and
  $S_n(y)$ if no element with these properties exists at that point of the
  construction.

  With the exception of the relation $S_n$, our structure looks exactly as the one
  constructed in~\cite[Theorem 2]{turetsky2020}. For $(F,\tau),(G,\rho)\in
  [\omega]^{<\omega}\times \omega^{<\omega}$ we have that:
  \begin{itemize}\tightlist
    \item ${W}^\S_\sigma((F,\tau))$  if and only if $\sigma=\tau$
    \item ${E_i}^\S((F,\tau),(G,\rho))$ if and only if $\tau=\rho$ and $F\Delta
      G=\{i\}$
    \item $P^\S((F,\rho),(G,\tau))$ if and only if $\tau=\rho\concat i$ for some $i$
      and one of the following holds:
      \begin{itemize}\tightlist
        \item $i\not\in F$ and $|G|$ is even or
        \item $i\in F$ and $|G|$ is odd.
      \end{itemize}
  \end{itemize}
  It is convenient to think of each $W_\sigma$ slice of $\S$ as an infinite
  dimensional hypergraph with the edge relations given by the $E_i$. One can
  then prove that the automorphisms of these slices in the reduct
  $(E_i)_{i\in\omega}$ are exactly the maps of the form
  $(F,\sigma)\mapsto(F\Delta H,\sigma)$ for some fixed $H\in
  [\omega]^{<\omega}$~\cite[Claim 2.1]{turetsky2020}. So, let $g$ be
  a non-trivial automorphism of the structure we have defined so far, then
  $g$ acts non-trivially on one of the $W^\S_\sigma$, i.e., $g:(F,
  \sigma)\mapsto (F\Delta H,\sigma)$ for $H\neq \emptyset$. In particular
  $g:(\emptyset,\sigma)\mapsto(H,\sigma)$. The $P$ predicate forces $g$ to
  act on $\sigma\concat i$ for $i\in H$. This can be used to prove that the
  automorphisms of this presentation of $\S$ are Muchnik equivalent to the
  paths through $T$.

  The only remaining issue is that $\S$ has many computable presentations and
  so far we have only controlled the automorphisms of one such presentation.
  This is where the $S_n$ predicates come in. They are used to ensure that
  $\S$ is computably categorical. This process is what requires a $\mathbf 0''$
  priority argument and thus introduces the $\emptyset''$ in the original
  construction. The construction will give rise to a $\mathbf 0''$ computable tree
  $Q\subseteq \omega^{<\omega}$ that is $\mathbf 0''$ isomorphic to $T$ such that for $\sigma\in Q$, $S_n((F,\sigma))$
  will hold for all $n\in\omega$ and $F\in [\omega]^{<\omega}$ and for
  $\sigma\not\in Q$, there is an $n$ such that $S_n((F,\sigma))$ holds if and
  only if $F=\emptyset$. One can show that the automorphisms of $\S$ are then
  Muchnik equivalent to the paths through $Q$. In particular, if $[Q]$ is a singleton, then there is only one non-trivial automorphism of $\S$. 

  However, $Q$ is only $\mathbf 0''$ isomorphic to the original tree
  $T$. This is why we modify the construction by forcing that $\sigma\in Q$
  with $\sigma(i)\downarrow$ satisfies $2\not\mid \sigma(i)$ if and only if $i\in
  \emptyset''$. Now, given an automorphism of $\S$ we can compute a path $f\in
  [Q]$ and from $f\oplus \emptyset''$ we can compute a path through $T$ by
  computing the isomorphism.
  However, the condition on $Q$ guarantees that $f\geq_T \emptyset''$ and thus
  $\{\nu:\nu\in (Aut(\S)\setminus \{id\})\}\geq_w \{f\oplus \emptyset'':f\in [T]\}$.
  That $\{\nu:\nu\in (Aut(\S)\setminus \{id\})\}\leq_w \{f\oplus
  \emptyset'':f\in [T]\}$ follows from a similar argument. Given
  $f\in [T]$, $f\oplus\emptyset''$ can compute an element of $[Q]$ and thus
  a non-trivial automorphism of $\S$.
  
  \emph{Construction.}
  Turetsky's construction had a requirement $G$, requirements $N_\pi$ for every
  $\pi\in T$ and $M_i$ for every $i\in\omega$. We add requirements $R_i$ for
  $i\in\omega$ where $R_i$ has the aim to code the membership of $i$ in
  $\emptyset''$.

  We use Turetsky's order of requirements, interleaving our
  requirements $R_i$ such that $R_i\leq N_\pi$ for all $\pi$
  with $|\pi|=i+1$. The
  potential outcomes for $R_i$ are $\infty\leq fin$.
  Using the fact that $\emptyset''$ is $\Sigma^0_2$ we fix a computable
  function $f$ such that $x\in \emptyset''$ if and only if $W_{f(x)}$ is
  a proper initial segment of $\omega$ and $W_{f(x)}=\omega$ otherwise. We will also need to modify Turetsky's strategy for $N_\pi$ where
  $\pi\neq \emptyset$. 
  The goal of the strategy $M_i$ is to ensure that if $\tau$ is a strategy on
  the true path (i.e., the left most path of the priority tree visited
  infinitely often) and $\M_i$ -- the $i$th computable structure in a computable enumeration of the
  structures in the language of $\S$ -- is isomorphic to $\S$, then $\M_i$ is
  computably isomorphic to $\S$. This strategy is unmodified and we thus omit
  it here.

  At the end of every stage $s$, declare
  $S_k((\emptyset,\sigma))$ for every $k<s$ and $\sigma\in s^{<s}$ not chosen by any
  strategy.

  \emph{Strategy for $R_i$.} Suppose $\tau$ is a strategy for $R_i$ visited at stage
  $s$. Check if $W_{f(i),s}\supset W_{f(i),t}$ where $t$ is the last stage that
  $\tau$ acted if it exists. If so
  let the outcome of $\tau$ be $\infty$. Otherwise finish with outcome $fin$.

  \emph{Strategy for $N_{\pi\concat i}$.} Suppose $\tau$ is a strategy for the
  requirement $N_{\pi\concat i}$. Then there is a unique $\rho\subset \tau$
  such that $\rho$ is a strategy for $N_\pi$ and a unique $\sigma\subset\tau$
  such that $\sigma$ is a strategy for $R_{|\pi|}$. Suppose $s_0$ is the first
  stage such that $\tau$ is visited and that $\rho$ has declared $\zeta$ to be the image
  of $\pi$. If the outcome of $\sigma$ is $\infty$, then choose the least even $m>s_0$ not mentioned in
  the construction and declare $\zeta\concat m$ to be the image of $\pi\concat i$.
  Otherwise choose the least odd $m>s_0$ not mentioned in the construction and declare $\zeta\concat m$ to
  be the image of $\pi\concat i$.

  At every stage $s$ when $\tau$ is visited, declare $S_n((\emptyset,\zeta\concat m))$ for the
  least $n$ such that $S_n((\emptyset,\zeta\concat m))$ does not hold at this
  stage and declare $S_k((F,\zeta\concat m))$ for all $k<n$ and
  $F\subseteq\{0,\dots,s\}$. Take the outcome \verb+outcome+.

  \emph{Verification.} Let $P$ be the tree of strategies of this construction
  and let $T$ be the tree of strategies of \cite[Theorem 2]{turetsky2020}. Given
  $\sigma\in P$, define its reduct $\sigma_r\in T$ by deleting the occurences
  of
  $R$-strategies from $\sigma$. I.e., assume without loss of generality that
  no $R_i$ strategy is succeeded by any $R_j$ strategy for $i,j\in\omega$ and define 
  \[ k(-1)=-1 \text{ and } k(i)=\begin{cases} k(i-1)+1
  & \sigma(k(i-1)+1)\not\in \{\infty,fin\}\\
  k(i-1)+2 &\text{otherwise}\end{cases} \]
  for $i\geq 0$ and let $\sigma_r(i)=\sigma(k(i))$.
  Let $P^-=\left\{ \sigma\in P: \sigma(|\sigma|-1)\not\in
  \{\infty,fin\}\right\}$.
  While $P^-$ is not a set-theoretic tree, it still is a tree as a partial
  order. It is then easy to see that the map $h: P^{-}\to T: \sigma\mapsto
  \sigma_r$ is a homomorphism of trees that preserves the priority ordering,
  i.e., for $\sigma,\tau\in P^{-}$
  \[ \sigma\preccurlyeq\tau \implies h(\sigma) \preccurlyeq h(\tau) \text{ and
  } \sigma\leq^P \tau \implies h(\sigma)\leq^T h(\tau).\]
  \begin{claim}\label{claim:carryover}
    If a path $f\in [T]$ is visited infinitely often, then there is
    a path $g\in [P^-]=[P]$ such that $h(g)=f$. Moreover, if $f$ is
    least in $[T]$ with respect to the priority ordering, then $g$ can be taken 
    least in $[P]$.
  \end{claim}
  \begin{proof}
    Run a strategy $\sigma$ for $R_i$ in isolation and let the true outcome of
    $\sigma$ be the leftmost outcome of $\sigma$ that occurs infinitely often.
    Clearly the true outcome is $\infty$ if and only if $W_i$ is infinite and
    $fin$ otherwise. Furthermore, it does not depend on the outcome of any other strategies.
    Denote this outcome by $t_i$. Given a strategy $\sigma\in T$ that is
    visited infinitely often note that for $\tau,\rho\in h^{-1}(\sigma)$,
    $|\tau|=|\rho|$. Consider the string $\tau\in h^{-1}(\sigma)$ obtained by
    filling the gaps by $t_i$, i.e., if $\tau(j)$ is the $i$th occurence of
    $\infty$ or $fin$ in $\tau$, then $\tau(j)=t_i$. Clearly, $\tau$ is visited infinitely often. Also, no string that was visited infinitely often
    in $h^{-1}(\sigma)$ can be to the left of $\tau$ as they can only differ on
    positions containing $\infty$ or $fin$. The claim now follows by induction
    on the length of $\tau$.
  \end{proof}
  By \cref{claim:carryover} we have that if $f$ is the true path in our
  construction, then $h(f)$ is the true path in the original construction. We
  have to prove that the strategies fulfil their goals, i.e., the $M_i$
  strategy ensures that if $\M_i\cong \S$, then $\M_i$ is computably isomorphic
  to $\S$. The verification of this is exactly as in Turetsky's original proof
  and thus omitted here.
  We also have to show that the $N_\pi$ strategies together ensure that we compute
  isomorphisms through a tree $Q$ isomorphic to $T$. At last we have to
  prove that every automorphism computes $\emptyset''$. This is done via the
  following claims.
  Let $f$ be the true path through the
  construction, that is the lexicographically least path in the priority tree that is visited
  infinitely often during the construction.
  \begin{claim}\label{thm_aut:claim_coding}
    Let $(F,\sigma)\in \S$ with $F\neq \emptyset$. Then $S_n^\S(F,\sigma)$ for all $n\in\omega$ if and only if 
    \begin{enumerate}\tightlist
      \item\label{it:turclaim1} $\sigma$ is the image of $\pi$ as declared by some $N_\pi$ strategy
        $\tau\subset f$,
      \item for $i\in \omega$ such that $\pi(i)\downarrow$, $W_{f(i)}=\{1\dots n\}$ if $2\nmid\sigma(i)$ and $W_{f(i)}=\omega$
        otherwise.
    \end{enumerate}
  \end{claim}
  \begin{proof}
    If $\sigma$ is the image of $N_\pi$ as declared by the unique $N_\pi$
    strategy $\tau\subset f$, then $S_n^\S(F,\sigma)$ for all $F\subseteq [\omega]^{<\omega}$ and
  $n\in\omega$.

  On the other hand, if $S_n^\S(F,\sigma)$ for all $n\in\omega$, then
  there is a unique $N_\pi$ strategy $\tau$ with $\sigma$ declared the image of
  $\pi$ that acts
  infinitely often. We have that $\tau\subset f$ because if $\tau$ was to the right
  of the true path then it could grow a single $\sigma$ only finitely many
  times before being initialized. Notice that by induction the map $\pi\mapsto
  \sigma$ mapping $N_\pi$ strategies to their images $\sigma$ is Lipschitz.
  Thus $\sigma(i)\downarrow$ for all $i$ such that $\pi(i)$ is defined.
  Furthermore, there are unique $\rho_i\subseteq \tau$ such that $\rho_i$ is
  a strategy for $R_i$. By $\tau$ being on the true path $\rho_i$ has outcome
$fin$ if and only if  $W_{f(i)}=\{1\dots n \}$ for some $n\in\omega$ and $\rho_i$
has outcome $\infty$ if and only if $W_{f(i)}=\omega$. By construction $2\nmid \sigma(i)$ in the first case and $2\mid \sigma(i)$ in the second
case.
  \end{proof}
  \begin{claim}
    For every non-trivial automorphism $g$ of $\S$, $g\geq_T \emptyset ''$.
  \end{claim}
  \begin{proof}
    By \cite[Claim 2.1]{turetsky2020} for every automorphism $g$ and every
    $\sigma$ there is $H_\sigma\in [\omega]^{<\omega}$ such that $g((F,\sigma))=(F\Delta
    H_\sigma,\sigma)$ for all $F\in [\omega]^{<\omega}$.
    If $g$ is non-trivial then there must be
    $\sigma$ such that $g((\emptyset,\sigma))=(H_\sigma,\sigma)$ for some
    $H_\sigma\neq
    \emptyset$ with $S_n^\S((H_\sigma,\sigma))$ for all $n\in\omega$. Now, by
    \cref{thm_aut:claim_coding}, $\sigma$ codes an initial segment of $\emptyset ''$. We can now 
    compute $\emptyset''$ iteratively as follows. Pick the least $i\in H_\sigma$. We have that
    $P^\S((\emptyset, \sigma),(\emptyset,\sigma\concat i))$ but
    $\neg P^\S((H_\sigma,\sigma),(\emptyset,\sigma\concat i))$, so in particular
    $g((\emptyset,\sigma\concat i))=(H_{\sigma\concat i},\sigma\concat i)$ for
    some $H_{\sigma\concat i}\neq
    \emptyset$. As $S_n^\S(\emptyset,\sigma\concat i)$ for all $n$, also
    $S_n^\S(H_{\sigma\concat i}, \sigma\concat i)$ for all $n$. Thus, the
    strategy declaring $\sigma\concat i$ has been visited infinitely often and
    so $\sigma\concat i\in Q$. By \cref{thm_aut:claim_coding}, $2|i$ if and only if
    $i\not\in \emptyset''$. Iterate this procedure to compute initial segments
    of $\emptyset''$ of arbitrary length and thus $\emptyset''$.
  \end{proof}
  Let $\sigma_\tau$ be the image of $\tau$ as declared by some $N_\tau$
  strategy on $f$ and $\sigma_\rho$ be the image of $\rho$ as declared by some
  $N_\rho$ strategy on $f$. Then by choice of $\sigma_\tau$, $\sigma_\rho$
  during the construction $\sigma_\tau\preccurlyeq \sigma_\rho$ if and only if
  $\tau\preccurlyeq \rho$. Let $Q$ be the set of images of the $N_\pi$
  strategies, then $Q$ is a tree isomorphic to $T$. Notice that $\mathbf 0''$
  can compute the true path and thus $\{f\oplus \emptyset'':f\in [T]\}\equiv_w
  [Q]$. To finish the proof it remains to show the following.
  \begin{claim}
    $|\{\nu:\nu\in (Aut(\S)\setminus \{id\})\}|=1$ if and only if $|[Q]|=1$.
  \end{claim}
  \begin{proof}
    Again, by~\cite[Claim 2.1]{turetsky2020} for every automorphism $g$ and every
    $\sigma$ there is $H_\sigma\in [\omega]^{<\omega}$ such that $g((F,\sigma))=(F\Delta
    H_\sigma,\sigma)$ for all $F\in [\omega]^{<\omega}$. In the proof of~\cite[Claim 2.2]{turetsky2020}, Turetsky constructs for each $i\in H_\sigma$ a path $f$ such that $\sigma\concat i\subset f$. So, if $|[Q]|=1$ there can be at most one non-trivial automorphism of $\S$ and if there is an automorphism then $[Q]$ can not be empty. On the other hand, Turetsky demonstrates for each $f\in [Q]$ and each $\sigma\concat i \subset f$ that there is an automorphism of $\S$ with $(\emptyset, \sigma)\mapsto (\{i\},\sigma)$. Thus, if $|\{\nu:\nu\in (Aut(\S)\setminus \{id\})\}|=1$, then $|[Q]|=1$ and if $[Q]$ is non-empty, then there must be a non-trivial automorphism of $\S$.
  \end{proof}
\end{proof}
Modifying the structure $\S$ slightly as in~\cite[Corollary 3]{turetsky2020} we can
obtain a proof of \cref{thm:rigiddgcat}.
\begin{proof}[Proof of \cref{thm:rigiddgcat}]
  Let $\S$ be the structure from \cref{lem:pathsinaut}. The structure $\S_1$ is
  obtained by adding two new elements $a_{odd}$ and $a_{even}$ to $\S$ and
  a new constant symbol $c$ to its vocabulary. The structure $\S_1$ satisfies
  \begin{itemize}
    \item $P^{\S_1}(a_{even}, (F,\langle\rangle))$ if and only if $|F|$ is
      even,
    \item $P^{\S_1}(a_{odd}, (F,\langle\rangle))$ if and only if $|F|$ is odd, 
    \item $c^{\S_1}=a_{even}$,
  \end{itemize}
  and no other relation holds on $a_{even}$ or $a_{odd}$. Let $\S_2$ be
  identical to $\S_1$ except that $c^{\S_2}=a_{odd}$. Turetsky proved that this
  structure has computable dimension 2 and that every path through $Q$ computes an
  isomorphism between $\S_1$ and $\S_2$ and vice versa. We will include this
  here for completeness and include a proof that this structure is rigid if
  $[Q]$ is a singleton. \cref{thm:rigiddgcat} then follows from the fact that
  $Q\cong T$ and if $[T]\geq_w \{\emptyset''\}$, then $[T]\equiv_w
  \{f\oplus \emptyset'': f\in [Q]\}\equiv_w [Q]$.

  Note that given $\B$ isomorphic to $\S_1$ via $f$, the substructure of $\B$ that does not
  contain $f(a_{even})$, $f(a_{odd})$ in the language of $\S$ is
  isomorphic to $\S$. As $\S$ is computably categorical, there is a computable
  isomorphism $g$ between this substructures. This isomorphism can be extended to an isomorphism between
  $\S_1$ and $\B$ by letting $g(a_{even})=f(a_{even})$ and
  $g(a_{odd})=f(a_{odd})$. If $c^\B=a_{even}$, this isomorphism is clearly
  computable. Otherwise, using a similar argument we obtain a computable
  isomorphism between $\B$ and $\S_2$.

  Notice that every isomorphism $g$ between $\S_1$ and $\S_2$ induces a non-trivial
  automorphism of $\S$ as $g(a_{even})=a_{odd}$ and hence $g((\emptyset,
  \langle\rangle))=(F,\langle\rangle)$ for some non-empty set $F$. But if $[T]$
  is a singleton, then $\S$ has exactly one non-trivial automorphism. Note,
  that in this case there is no non-trivial automorphism of $\S_1$, as the only
  such automorphism of $\S_1$ would map $(\emptyset, \langle\rangle)\mapsto
  (\{n\}, \langle\rangle)$ and hence $c^{\S_1}=a_{even}\mapsto a_{odd}\neq
  c^{\S_1}$. Thus, $\S_1$ is rigid.
\end{proof}

\subsection{Treeable degrees}\label{sec:treeable}
\begin{definition}[\cite{bazhenov2018}]
  For any computable structure $\S$, the \emph{spectral
  dimension} of $\S$ is the least $k\leq \omega$ such that there exists an
  enumeration $(\A_i,\B_i)_{i<k}$ of pairs of computable copies of $\S$ such
  that 
  \[ CatSpec(\S)=\bigcap_{i<k}\{ deg(X): (\exists f\leq_T X) f:\A_i\cong \B_i\}.\]
  Notice that a degree $\mathbf d$ is a \emph{strong degree of categoricity} if there is
  a structure $\S$ of spectral dimension $1$ with degree of categoricity
  $\mathbf d$.
\end{definition}

Recall that $P\subseteq\omega^\omega$ is a $\Pi^0_1$ class if there is
a computable tree $T$ with $P=[T]$.
\begin{definition}
  A Turing degree $\mathbf d$ is \emph{treeable} if there is a $\Pi^0_1$ class
  $P$ such that there exists $p\in P$ with $deg(p)=\mathbf d$ and $P\geq_w
  \{p\}$.
\end{definition}

\begin{lemma}\label{lem:dgcattreeable}
  Every strong degree of categoricity is treeable.
\end{lemma}
\begin{proof}
  Assume that $\A_1\cong \A_2$ witness that $\mathbf d$ is a strong degree of
  categoricity. The isomorphisms between these two structures form a $\Pi^0_1$
  class and thus there is a computable tree $T$ with $[T]=\{ f: f:\A_1\cong
\A_2\}$. By definition of degrees of categoricity the Turing-least element of this class is of degree $
\mathbf d$.
\end{proof}
Clearly, if a structure $\S$ has computable dimension $2$, then it has spectral
dimension $1$ and thus any degree of categoricity of a structure with computable dimension $2$
must be strong. We thus obtain the following from \cref{thm:rigiddgcat}.
\begin{corollary}\label{cor:dgcatifftreeable}
  Let $\mathbf d\geq \mathbf 0''$. Then $\mathbf d$ is a strong degree of categoricity
  if and only if it is treeable.
\end{corollary}
\begin{question}
  Is every treeable degree the degree of categoricity of a structure?
\end{question}
\begin{question}
  Is every degree of categoricity treeable?
\end{question}

Which degrees are treeable? Recall that $f\in \omega^\omega$ is a $\Pi^0_1$
function singleton if there is a tree $T$ with $[T]=\{f\}$. Clearly degrees that contain $\Pi^0_1$ function
singletons are treeable degrees and one can find plenty of examples of those in
the literature. Let us summarize.

Abusing notation we say that a degree $\mathbf d$ is a $\Pi^0_1$ function singleton if it contains
a $\Pi^0_1$ function singleton. We now list examples of classes of degrees that
are $\Pi^0_1$ function singletons and thus by \cref{thm:rigiddgcat} degrees of
categoricity of rigid structures with computable dimension $2$.

Most of these examples are obtained using the following trick due to Jockusch
and McLaughlin~\cite{jockusch1969}. Let $X$ be
a $\Pi^0_2$ singleton, i.e., the only solution to a $\Pi^0_2$ predicate. 
Then it is defined by a formula $\forall u\exists v R(u,v,X)$ which can be
rewritten as
\[ \forall u \exists (v>u) \forall (x<v) R(u,v,\chi_X\restrict x) \]
where $R(u,v,\chi_X\restrict x)$ implies $R(u,v,\chi_X\restrict y)$ for $y<x$.
Let $f(u)$ be the string $\sigma\preceq \chi_X$ of length $\mu
v[R(u,v,\sigma)]$. Then $f\leq_T X$ by definition and $X\leq_T f$ as the length
of the strings in the range of $f$ is unbounded. The function $f$ itself is
a $\Pi^0_1$ singleton as it is the unique solution to the equation
\[ \forall x (f(x)\in 2^{<\omega} \land R(x,|f(x)|,f(x)))\land \forall
(\tau\preceq f(x)) \neg R(x,|\tau|,\tau).\]

So, a degree $\mathbf d$ contains a $\Pi^0_1$ function singleton if and only
it contains a $\Pi^0_2$ singleton. The following theorem is well known.
\begin{lemma}[{folklore, see~\cite[Section
  II.4.]{sacks1990}}]\label{lem:jumpfuncsing}
  For all computable ordinals $\alpha$, $\mathbf 0^{(\alpha)}$ is a $\Pi^0_1$
  function singleton.
\end{lemma}
\begin{lemma}[{folklore, see~\cite{odifreddi1999}}]\label{lem:betweenfuncsing}
  If $\mathbf d$ is a $\Pi^0_1$ function singleton, then so is every $\mathbf
  c$ with $\mathbf d\leq\mathbf c\leq \mathbf d'$.
\end{lemma}
\begin{proof}
Let $D\in\mathbf d$ be a $\Pi^0_2$ singleton defined via $\phi(X)$ and let $C\in \mathbf c$, then $C$ is limit computable from
$D$, say via the $\Delta^0_2$ function $\psi(X,y)$. Consider the set $C\oplus D$. It is the unique solution to the 
formula
\[ \phi(X_{odd})\land \left(\psi(X_{odd},y)\leftrightarrow 2y\in X\right)\]
where $X_{odd}=\{x: 2x+1\in X\}$.
Using the above trick we get that the degree $\mathbf c$ of $C\oplus D$ is a $\Pi^0_1$
function singleton.
\end{proof}
Combining \cref{lem:jumpfuncsing} for $\alpha\geq 2$ with
\cref{lem:betweenfuncsing} and \cref{thm:rigiddgcat} we get a new class of degrees that are
degrees of categoricities and answer a question posed by Csima and
Ng~\cite{csima}.

\begin{corollary}\label{cor:compd}
  Fix $\alpha>2$ and let $\mathbf d$ be such that $\mathbf 0^{(\alpha)}\leq \mathbf d\leq \mathbf
  0^{(\alpha+1)}$, then there is a computable rigid structure $\S$ with computable dimension $2$ such that $dgCat(\S)=\mathbf d$.
\end{corollary}
We can get even more. Harrington~\cite{harrington1976} exhibited two $\Pi^0_1$
function singletons $h_1$ and $h_2$ that are arithmetically incomparable. This
implies the existence of a $\Pi^0_1$ function singleton that is not of the same
arithmetic degree as $\mathbf 0^{(\alpha)}$ for any computable ordinal $\alpha$
and thus also the existence of a degree of categoricity of this sort.
In~\cite{harrington1976a}, Harrington observed even more (see also the
exposition in Gerdes~\cite[Section 3.4]{gerdes2010}).
\begin{lemma}\label{lem:harrington}
  For every computable ordinal $\alpha$, there is a uniform sequence of $\Pi^0_1$
  function singletons $(h_i)_{i\in\omega}$ such that for all $n$ and all
  $\beta<\alpha$,
  \begin{enumerate}
    \tightlist
    \item $h_n$ is generalized $low_\beta$, i.e., ${h_n}^{(\beta)}\equiv_T
      h_n\oplus \emptyset^{(\beta)}$,
    \item $\forall X \left(X\leq_T\emptyset^{(\alpha)} \land X\leq_T
      \emptyset^{(\beta)} \implies X\leq_T \emptyset^{(\beta)}\right)$,
    \item and $h_n\not\leq_T \left(\bigoplus_{i\neq n} h_i\right)^{(\alpha)}$.
  \end{enumerate}
\end{lemma}
While the degrees of the $h_n$ might not satisfy the conditions of
\cref{cor:dgcatifftreeable} we can take ${h_n}''$ to obtain a variation of
the result of Csima and
Stephenson~\cite{csima2018}, that exhibits a degree of categoricity computable
from $\mathbf 0''$ that is not in
$[\mathbf 0^{(\alpha)}, \mathbf 0^{(\alpha+1)}]$ for any computable ordinal
$\alpha$.
\begin{corollary}\label{cor:notalphacomp}
  For every computable ordinal $\alpha$, there is a degree of categoricity
  $\mathbf d$ such that $\mathbf d\not\leq_T \mathbf 0^{(\alpha)}$ and $\mathbf
  d\not\in [\mathbf 0^{(\gamma)},\mathbf 0^{(\gamma+1)}]$ for any computable
  ordinal $\gamma$.
\end{corollary}
In particular, combining \cref{lem:harrington}, the relativized version of \cref{lem:jumpfuncsing} and
\cref{lem:betweenfuncsing} we can get a vast amount of new degrees of
categoricity outside 
the spine of jumps of $\mathbf 0$.
\begin{corollary}\label{cor:nonarithcompd}
  For any computable ordinal $\alpha$, there exist degrees $(\mathbf
  d_i)_{i\in\omega}$, such
  that for all $n$ and for all $\beta<\alpha$, $\mathbf
  d_n\not\leq\bigoplus_{i\neq n} {\mathbf
{d}_i}^{(\alpha)}$, ${\mathbf{d}_n}^{(\beta)}\not\geq \mathbf 0^{(\alpha)}$,
  and for all $\mathbf c\in [{\mathbf{d}_n}^{(\beta)},{\mathbf{d}_n}^{(\beta+1)}]$, $\mathbf c$ is the
  degree of categoricity of a rigid structure with computable dimension $2$.
\end{corollary}

It is also not hard to give examples of degrees that are not treeable
by looking at sufficiently generic degrees. This is not surprising as examples
of degrees that are not degrees of categoricity are often obtained by
building degrees that are sufficiently
generic~\cite{anderson2012,franklin2014}.

\begin{lemma}\label{lem:genericnottreeable}
  Let $\mathbf d$ be sufficiently generic, then $\mathbf d$ is not treeable.
\end{lemma}
\begin{proof}
  Let $T$ be a tree with Turing-least element of degree $\mathbf d$ and let 
  \begin{equation}\label{eq:pinsingleton} \phi(x)\iff \exists i \forall
      n\, \Phi_i^x\restrict n\in T \land \forall i (\forall n\,\Phi_i^x\restrict n \in T \to x\leq_T
    \Phi_i^x)\end{equation}
  Since $x\leq_T \Phi_i^x$ is $\Sigma^0_3$, $\phi(x)$ is $\Pi^0_4$.
  Furthermore, $\phi(f)$ if and only if $deg(f)=\mathbf d$. So, say $g\in
  \mathbf d$ is $4$-generic, then there is $\sigma \subset g$ such that $\sigma\Vdash
  \phi(x)$. But then for all generic $f\supset \sigma$ $\phi(f)$ and we can choose such
  $f$ to be Turing incomparable with $g$. Thus $\mathbf d$ is not treeable.
\end{proof}
\begin{proposition}
  For every computable ordinal $\alpha$, there is a hyperarithmetic degree
  $\mathbf d$ with $\mathbf d\geq  \mathbf 0^{(\alpha)}$ that is not a degree
  of categoricity.
\end{proposition}
\begin{proof}
  Let $\psi(x)$ be given similar to $\phi(x)$ as defined in
  \cref{eq:pinsingleton} but replace $x$ with $x^{(\alpha)}$. Let $\mathbf c$
  be a hyperarithmetic $(\alpha+4)$-generic degree, then $\mathbf c^{(\alpha)}$ is not a degree of
  categoricity as it is not treeable by the same argument as in
  \cref{lem:genericnottreeable} with $\psi(x)$ in place of $\phi(x)$.
\end{proof}
Note that by \cref{eq:pinsingleton} in the proof of
\cref{lem:genericnottreeable} every treeable degree is a countable $\Pi^0_4$ class.
Tanaka~\cite{tanaka1972} showed that every countable $\Sigma^0_{n+1}$ class contains
a $\Pi^0_n$ singleton. So, in particular, if $\mathbf d$ is treeable, then it
contains a $\Pi^0_4$ singleton. 
\begin{proposition}
  Every strong degree of categoricity contains a $\Pi^0_4$ singleton.
\end{proposition}

\section{Degrees of categoricity in $[\mathbf 0',\mathbf 0'']$}\label{sec:csimang}
In this section we show that all degrees between $\mathbf 0'$ and $\bf{0''}$
are degrees of categoricity. In \cite{fokina2010}, Fokina, Kalimullin and
Miller showed that every degree d.c.e.\ in and above $\mathbf{0}^{(n)}$ for
some $n$ is a degree of categoricity. They did this by noting that their proof
that every d.c.e.\ degree is a degree of categoricity relativises to give that every degree d.c.e. in and above a given degree is a degree of categoricity \emph{relative to the given degree}, and then using Marker extensions to get the desired result. Now that Csima and Ng have shown that every $\Delta^0_2$ degree is a degree of categoricity in \cite{csima}, we can verify that the methods of Fokina, Kalimullin and Miller also apply in this case, and obtain the desired result.

To begin, we give the definition of a categoricity spectrum relative to a degree.
\begin{definition}
	Let $\tau$ be a computable vocabulary, let $\bf{c}$ be a Turing degree, let $\A$ be a $\bf{c}$-computable $\tau$-structure,
	and let $(\B_e)_{e\in\omega}$ be an enumeration of all $\bf{c}$-computable
	$\tau$-structures. The \emph{categoricity spectrum of $\A$ relative to $\bf{c}$}
	is the set
	\[CatSpec_{\bf{c}}(\A)=\bigcap_{e\in \omega: \B_e\cong \A}\{deg(X): (\exists
	f:\A\cong\B_e) X\geq_T f\}.\]
\end{definition}

The proof given in Theorem 5.9 of \cite{fokina2010} shows that if ${\bf
d}\geq_T {\mathbf{0}^{(m)}}$, and if there exists
a $\mathbf{0}^{(m)}$-\emph{acceptable} structure $\mathcal{M}$ with
$CatSpec_{\mathbf{0}^{(m)}}(\mathcal{M})$ the upper cone above ${\bf d}$, then
${\bf d}$ is a strong degree of categoricity.
A $\mathbf{0}^{(m)}$-\emph{acceptable} structure is one with
a $\mathbf{0}^{(m)}$-computable copy in which all predicates have an infinite
computable subset. So it remains to argue that if $\mathbf 0'<\mathbf d<\mathbf
0''$, then there is a  ${\mathbf{0}^{(1)}}$-\emph{acceptable} structure
$\mathcal{M}$ with $CatSpec_{\mathbf{0}^{(1)}}(\mathcal{M})$ the upper cone above ${\bf d}$.

We now consider the proof given in \cite{csima} that every $\Delta^0_2$ degree
is a strong degree of categoricity. It proceeds by considering a $\Delta^0_2$
appoximation to a set $D \in {\bf d}$, an effective list of the possible
computable structures, and builds the structures $\A$ and $\hat{\A}$ that
witness ${\bf d}$ as a strong degree of categoricity in a stage-by-stage
construction. Though the proof as described in the paper uses infinitely many
relation symbols, it is noted that this is just a convenience - the structure
can be easily converted into a graph using standard codings. Running the
construction relative to an oracle $\mathbf 0'$, for a degree ${\bf d} \geq
{\bf0'}$ that is $\Delta^0_2$ in ${\bf 0'}$ (i.e., $\mathbf 0'\leq\mathbf d\leq
\mathbf 0''$), and adjoining the standard copy of a structure with strong
degree of categoricity $\mathbf 0'$ then yields a structure with strong degree
of categoricity ${\bf d}$ relative to $\mathbf 0'$. Moreover, when we transform
the presentation into a graph, we easily obtain the infinite computable subset
of the edge relation required to guarantee an acceptable copy of the structure
- they occur as we mark off the independent modules where the different
requirements will act. Combining this with \cref{cor:compd} and Csima and Ng's
result we obtain the following.
\begin{theorem}\label{thm:delta2relativized}
  Every degree $\mathbf d$ with $\mathbf
  0^{(\alpha)}\leq\mathbf d\leq \mathbf 0^{(\alpha+1)}$ for some computable
  ordinal $\alpha$ is the degree of categoricity of a computable structure.
\end{theorem}

\section{Not the degree of categoricity of a rigid
structure}\label{sec:bazhyama}
\cref{cor:compd} shows that for $\alpha>2$, every degree $\mathbf d$ between $\mathbf
0^{(\alpha)}$ and $\mathbf 0^{(\alpha+1)}$ is the degree of categoricity of
a rigid structure. Bazhenov and Yamaleev showed that there is a $2$-c.e.\
degree that is not the degree of categoricity of a rigid
stucture~\cite{bazhenov2017d}. In this section we prove the following theorem.
\begin{theorem}\label{thm:notrigid}
  There is a degree $\mathbf d$, $\mathbf 0'<\mathbf d<\mathbf 0''$, that is
  not the degree of categoricity of a rigid structure.
\end{theorem}
\cref{thm:notrigid} shows a limitation to Turetsky's technique of coding paths through trees into the
automorphisms of a structure: If one were to code treeable degrees below $\mathbf
0''$, one would not be able to preserve the 1-1 correspondence between automorphisms and
paths through the tree.

The main ingredient of the proof of \cref{thm:notrigid} is the following lemma.
Its proof combines Bazhenov and Yamaleev's construction with a true stage argument.
\begin{lemma}\label{lem:notstrongrigid}
  There exists a degree $deg(D)$, $\mathbf 0'<deg(D)<\mathbf 0''$ that
  satisfies the following: for all computable structures $\A$ and $\B$ and all
  $\Delta^0_3$ isomorphisms $g:\A\cong \B$, $deg(D)=deg(g)$ implies that there
  is a computable structure $\mathcal N$ and an isomorphism $f:\A\cong \mathcal
  N$ with $deg(f)\not\leq deg(D)$.
\end{lemma}
Notice that \cref{lem:notstrongrigid} already shows that there is a degree
$\mathbf d\in (\mathbf 0',\mathbf 0'')$ that is not the strong degree of
categoricity of any rigid structure.  
\begin{proof}[Proof of \cref{lem:notstrongrigid}]
Let $g\leq \mathbf 0''$, then $g$ is limit computable in $\mathbf 0'$. Assume
that we have an approximation to $\emptyset'$ in terms of finite strings
$\sigma_i$ where
\[|\sigma_i|=\min\{\mu x [x\enters \emptyset'_i], i\}\quad\text{ and
}\quad  (\sigma_i(x)\downarrow \Rightarrow \sigma_i(x)=\emptyset'_i(x)).\]
Let the triple $(\A,\B,g)$ consist of computable graphs $\A$,
$\B$ and a $\Delta^0_2(\emptyset')$ function $g$. Then we may approximate
this triple via the triples $(\A[s],\B[s],g[s])$ where $\A[s]$ and $\B[s]$
are substructures of $\A$, respectively $\B$, with universe $\{0,\dots,s\}$, and, say that
$g$ is limit computable via $\Phi_e^{\emptyset'}$, then
$g[s](x)=\Phi_e^{\sigma_s}(x,t)$ where $t=\max \{ r<s:
\Phi_{e,s}^{\sigma_s}(x,r)\downarrow\}$. Note that it is not necessarily the case that $\lim_{s\to\infty}
g[s]=g$. However, if $i_0,\dots$ is a sequence of true stages in our
approximation to $\emptyset'$, then $\lim_{s\to\infty} g[i_s]=g$.
We will work with an enumeration of the approximations of all such triples
$(\A_e,\B_e,g_e)_{e\in\omega}$. Our construction will build a set $D$,
$\emptyset'\leq_T D\leq_T\emptyset''$ and ensure that if
$g_e:\A_e\cong \B_e$ and $g_e\equiv_T D$, then there is a computable structure $\mathcal N_e\cong
\A_e$ and an isomorphism $f_e:\mathcal N_e\cong\A_e$  such that $f_e\not\leq_T
D$. That is, for every $e$ we aim to satisfy the following infinite set of requirements
\[ S_{e,i}:\quad \left( g_e: \A_e\cong \B_e \land g_e\in \Delta^0_3 \land
D=\Phi_e^{g_e} \land g_e=\Theta_e^D\right) \Ra f_e\neq \Psi_i^D\]
and the global requirement 
\[ G: \emptyset'\leq_T D.\]
The requirement $G$ is satisfied by building $D=\hat D\oplus \emptyset'$. 

Given $\sigma,\tau\in2^{<\omega}$ we say that $\tau$ is
\emph{$\sigma$-true}, if $\tau\preceq \sigma$.
We call a stage $t$ an $s$-true stage if the stage $t$ approximation $\sigma_t$ is 
$\sigma_s$-true. The key difference of our construction to Bazhenov and
Yamaleev's construction is that at stage $s$ we
only consider the work of strategies acting at $\sigma_s$-true stages to build
$D$.

For people who have seen that construction it might be
easy to see that satisfying $\mc S_{e,i}$ using true stages gives a set that is
$\Delta^0_2(\emptyset')$ and can not be the degree of categoricity of a rigid
structure.

The difficulty arises from coding $\emptyset'$. The key argument for the coding
is that if we take
$D=\hat D\oplus \emptyset'$, then the computation $\Psi_i^D$ might only change
finitely many times at any $x$ given that we preserve finite injury of $\hat
D$ on the true path. This is because at some point in the approximation
of $\emptyset'$ no bit lower than the use of $\Psi_i^D(x)$ will change anymore.
Thus, if at that point we have ensured that on the true path $\Psi_i^D(x)\neq f_e$, then this
will hold in the limit. However, with more complicated sets, such as
$\emptyset''$ this is not true anymore.

Let $i_0,\dots$ be the sequence of true stages for $\emptyset'$.
Together with $D=\lim_{j\to \omega} D[i_j]$ we will build
$\mathcal N_e=\lim_{j\to\omega}\mathcal N_e[j]$. Say that a stage $s$ in the
construction is $e$-expansionary if $\mathcal N_e[s+1]\neq \mathcal N_e[s]$. An
$e$-expansionary stage is always caused by some strategy that tries to satisfy
$S_{e,i}$. Let $s$ be an $e$-expansionary stage, then $\mathcal
N_e[s]$ will have universe $\{0,\dots, s\}$ and will embed into $\mc A_e$.
Say we are at some stage $s$, and $t<s$ is $e$-expansionary.
The strategy acting for a requirement $S_{e,i}$ might force us to change the
embedding $f_e[t]$ because $g_e[s]\neq g_e[t]$. There are two ways we can do
this.
\begin{center}
\begin{tikzcd}
  \A_e[s]\ar[rr,bend left, "{g_e[s]}"]& \A_e[t]\ar[r, "{g_e[t]}"]\ar[l,"id"']& \B_e[s]\\
  \mc{N}_e[s]\ar[u,"{id \circ f_e[t]\circ id^{-1}}"',dashed, bend right]\ar[u,"{g_e[s]^{-1}\circ
g_e[t]\circ f_e[t]\circ id^{-1}}" ,dotted, bend
  left]&&\mc{N}_e[t]\ar[ul,"{f_e[t]}"']\ar[ll,"id"] 
\end{tikzcd}
\end{center}
We can extend any partial embedding of $\mc N_e[s]$ into $\A_e[s]$ by
recursively mapping elements not in the domain to the least element not in the
range as $N_e[s]=A_e[s]=\{0,\dots,s\}$. We say that $f_e[s]$ is the
\emph{id-extension} of $f_e[t]$ if it is obtained by extending $id \circ f_e[t]\circ
id^{-1}$ and $f_e[s]$ is a \emph{g-extension} of $f_e[t]$ if it is obtained by
extending $g_e[s]^{-1}\circ
g_e[t]\circ f_e[t]\circ id^{-1}$. Note that we can take \emph{g-extensions}
because technically $g_e[t]:\A_e[t]\ra\B_e[t]$ and
$\B_e[t]\subseteq \B_e[s]$.

We use upper case greek letters to denote Turing operators and the
corresponding lower case greek letters to denote their use. We abuse notation
and do not write the oracle in the use function as the oracle will be implied
from the context. As we are only interested in total functions we can make the
following assumption on uses without loss of generality: If
$\Phi^X(x)\downarrow$, then for all $y<x$, $\Phi^X(y)\downarrow$ and
$\phi(y)<\phi(x)$.
During the construction we will use the following length agreement functions:
\begin{multline*} L(e)[s]=\max \bigg\{x : \forall (y\leq x)\Big[
      \Phi_e^{g_e}(y)[s]=D[s](y)\land\\
      \forall (z<\phi_e(y)[s]) \left( \Theta_e^{D}(z)[s]=g_e(z)[s] \land
      (\{0,\dots, \phi_e(y)[s]\}\subseteq\A_e[s])\right)\Big]\bigg\}
  \end{multline*}
  \[ l(e,i)[s]=\max \{ x: \forall (y\leq x)(f_e(y)[s]=\Psi_i^{D}(y)[s])\}\]

The priority ordering on our requirements is given by $S_{e,i}< S_{e',i'}$ if and only if $\langle e,i\rangle <\langle e',i'\rangle$. If the
above is true, we say that $S_{e,i}$ has higher priority than $S_{e',i'}$. 
Assume we are at stage $s$ of the construction and let $t$ be
the last $s$-true stage. At the beginning of $s$ we let $D[s]=D[t]$. A strategy $S_{e,i}$ might require attention because of the following reasons:
\begin{enumerate}
  \item Its witness $x_{e,i}$ is undefined.
  \item $x_{e,i}\not\in D[s]$, $x_{e,i}<L(e)[s]$, and $f_e(y)$ is
    undefined for some $y<\phi_e(x_{e,i})[s]$.
  \item $x_{e,i}\not\in D[s]$, $x_{e,i}<L(e)[s]$, and $\phi_e(x_{e,i})[s]<l(e,i)[s]$.
  \item $x_{e,i}\in D[s]$ and $x_{e,i}< L(e)[s]$.
\end{enumerate}
\paragraph*{Construction.} At stage $0$ define $D[0]=\emptyset$, $\mc N_e=\emptyset$
and $f_e=\emptyset$ for all $e\in\omega$.
Say we are at stage $s>0$ of the construction. Let $D[s]=D[t]$ and then
enumerate $2x+1$ into $D[s]$ for all $x$ such that $x\in \emptyset'_{s}$. Find
the least requirement that requires attention and initialize all lower priority
strategies. Then choose the following depending on what caused the attention.
\begin{enumerate}
  \item Define $x_{e,i}$ to be an even number larger than all numbers used in
    the construction so far.
  \item Let $s^-$ be the last
    $e$-expansionary stage. Define $\mc N_e[s]$ and $f_e[s]$ as the $id$-extension of $f_e[s^-]$.
  \item Enumerate $x_{e,i}$ into $D[s]$.
  \item Extract $x_{e,i}$ from $D[s]$. Let $s^-$ be the last $e$-expansionary
    stage. Define $\mc N_e[s]$ and $f_e[s]$ as the $g$ extension of $f_e[s^-]$.
\end{enumerate}
\paragraph*{Verification.}
We first verify the construction for a single requirement $S_{e,i}$ in
isolation. Say $S_{e,i}$ acts for the first time at stage $s_0$ and picks
$x_{e,i}$. Notice that as $S_{e,i}$ will never be initialized it won't ever
act again because of reason (1). Now, say that $s_1>s_0$ is a true stage such that $\emptyset'_{s_0}\restrict
x_{e,i}/2=\emptyset'\restrict x_{e,i}/2$. By our isolation assumption the bits below $x_{e,i}$
in our approximation to $D$ won't change anymore after $s_1$.

Say $S_{e,i}$ needs attention because of reason (2) at true stage $s_2>s_1$. We
do an $id$-extension of $f_e[s^{-}]$. Such an extension can always be done
because $f_e:\mc N_e\to \A_e[s^-]$ and $\A_e[s^-]\subseteq \A_e[s_2]$. If
$S_{e,i}$ never needs attention again we have satisfied $S_{e,i}$. Now, say
that $s_3$ is a true stage such that $S_{e,i}$ receives attention. Then this
can not be because of reason (2) because
if $x_{e,i}<L(e)[s_3]$, then $\Theta^D_e(z)[s_3]=\Theta^D_e(z)[s_2]=g_e[s]$ for all
$z<\phi_e(x_{e,i})[s]$ and so $\phi_e(x_{e,i})[s_2]=\phi_e(x_{e,i})[s_3]$ and
we have already extended $f_e$ at
$s_2$. Thus $S_{e,i}$ must act because of reason (3) and so $x_{e,i}$ is
enumerated into $D$ at stage $s_3$. Notice that, since
$s_3$ is a true stage, for every $t\geq s_3$, $s_3$ is $t$-true. Let $s_4$ be
the next true stage such that $S_{e,i}$ receives attention. Notice that this
must be because of reason (4) as $x_{e,i}\in D[s_4]$. We extract $x_{e,i}$
from $D$ and do a $g$-extension hoping to diagonalize against $\Psi_i^D$
computing $f_e$. It might be the case that there is a stage $t$ (not
necessarily a true stage), $s_3<t<s_4$ such that $t=s_4^-$, $g_e[t]\subseteq
g_e[s_4]$ and the $g$-extension of $f_e[t]$ is an $id$-extension. Say $S_{e,j}$
does the $g$-extension at $t$, then $S_{e,j}$ is not of higher priority than
$S_{e,i}$ as otherwise $S_{e,i}$ would have been reset. But then
$L(e)[t]>x_{e,j}>x_{e,i}$ and thus $S_{e,j}=S_{e,i}$. So, after doing the
$g$-extension at $s_4$ we have that
\[ f_e[s_4]\supseteq f_e[t]\neq f_e[s_3] = \Psi_i^D[s_3]\subseteq\Psi_i^D[s_4].\]
So we have successfully diagonalized and given that $D$ does not change below
$x_{e,i}$ anymore $\Psi_i^D\neq f_e$.

We have that $x_{e,i}>L(e)[s_4+1]$. This allows $g_e$ to change and
achieve that $L(e)[s_5]>x_{e,i}$ at some stage $s_5$. The risk is that
$\phi_e(x_{e,i})[s_5]<\phi_e(x_{e,i})[s_4]$ and that thus we might want to
take action because of $(3)$ again. Notice that this would imply that
$g_e$ changed below $\phi_e(x_{e,i})[s_4]$. But at $s_4$ we had that
\[\forall (z<\phi_e(x_{e,i})[s_4])(\Theta_e^{D}(z)[s_4]=g_e(z)[s_4]).\]
So this is particularly true for all $z<\phi_e(x_{e,i})[s_5]$. By our
assumption on use functions $\phi$ that $x<y$ implies $\phi(x)<\phi(y)$ we have
that
\[\theta_e(\phi_e(x_{e,i})[s_5])<\theta_e(\phi_e(x_{e,i})[s_4])\leq
  x_{e,i}\]
and as $D[s_4]\restrict x_{e,i}=D[s_5]\restrict x_{e,i}$,
\[ \forall (z<\phi_e(x_{e,i})[s_5])(
\Theta_e^{D}(z)[s_4]=g_e(z)[s_5]=g_e(z)[s_4])\]
contradicting that $g_e$ changed below $\phi_e(x_{e,i})$. Notice that a true
stage $t$ is $s$-true for all $s>t$. Thus, in isolation, $S_{e,i}$ will never
require attention after stage $s_4$.

We now drop the isolation assumption. Notice first that if $t>s_4$ and
$f_e[t]\restrict \phi_e(x_{e,i})[s_4]\neq f_e[s_4]\restrict
\phi_e(x_{e,i})[s_4]$, then by the argument in the above paragraph this can only have happened because of
a $g$-extension of a higher priority argument which would cause $S_{e,i}$ to
initialize. Let us show that no
requirement is initialized infinitely often. Recall that in isolation every requirement
requires attention at most $4$ times on the true path. 
Using the fact that every true stage $t$ is
$s$-true for all $t>s$ we can see by induction that for every requirement there
is a stage after which it is not initialized again. By the same argument, there
is a stage $t$ such that if $S_{e,i}$ requires attention at $s>t$, then
it is the highest priority requirement needing attention and thus will receive 
attention at $s$.

Notice, that if the lefthand side of the implication
in $S_{e,i}$ is true, then $\mc N_e$ is necessarily infinite and isomorphic to
$\A_e$ as for every $i$, $S_{e,i}$ receives attention because of condition (2)
at least once. Furthermore, no $\Psi_i^D$ can compute $f_e$, by our above
argument. So, every requirement $S_{e,i}$ is satisfied. The requirement $G$ is
trivially satisfied in the limit.
\end{proof}
To obtain the proof of \cref{thm:notrigid} it remains to show that every degree
of categoricity of a rigid structure is strong. This follows immediately from
the following proposition due to Bazhenov, Kalimullin, and
Yamaleev~\cite{bazhenov2020a}.
\begin{proposition}[{\cite[Proposition 1, item (b)]{bazhenov2020a}}]
Suppose $\mathbf d$ is the degree of categoricity of a rigid structure. Then
there exists a rigid structure $\S$ with degree of categoricity $\mathbf d$ and
spectral dimension $1$.
\end{proposition}
\printbibliography
\end{document}